\documentclass[10pt]{article}
\usepackage{amsmath,amssymb, amsthm, graphicx, amsopn}
\usepackage{arydshln}
\usepackage{enumerate, framed}
\usepackage[explicit]{titlesec}
\usepackage{nicefrac}

\usepackage{tikz}
\usetikzlibrary{matrix,arrows,decorations.pathmorphing}
\usepackage{tikz-cd}

\usepackage[utf8]{inputenc}
\usepackage[T1]{fontenc}
\usepackage[english]{babel}

\usepackage{tikz-cd}
\usetikzlibrary{decorations.markings}
\tikzset{negated/.style={
        decoration={markings,
            mark= at position 0.5 with {
                \node[transform shape] (tempnode) {$\backslash$};
            }
        },
        postaction={decorate}
    }
}

\newtheorem{theorem}{Theorem}[section]
\newtheorem{definition}[theorem]{Definition}

\newtheorem{proposition}[theorem]{Proposition}
\newtheorem{corollary}[theorem]{Corollary}

\newtheorem{example}[theorem]{Example}

\renewcommand{\qed}{\hspace*{\fill}\rule{1 ex}{1.5 ex}\\}

\let\<=\langle
\let\>=\rangle

\def\land{\wedge}       \def\lor{\vee}

\let\phi=\varphi
\let\theta=\vartheta
\let\epsilon=\varepsilon

\def\mindot{\buildrel\textstyle.\over
    {\hbox{\vrule height.55ex width0pt
        \smash{\hbox{\mathsurround=0pt$-$}}}}}

\def\TITLE{\bf\MakeUppercase{Subalgebra independence}}
\def\AUTHOR{\small\MakeUppercase{Alexa Gopaulsingh}, \and \small\MakeUppercase{Zal\'an Gyenis}, \and \small\MakeUppercase{\"Ovge \"Ozt\"urk}\normalsize}
\def\DATE{\today}
\def\ABSTRACT{Subobject independence as morphism co-possibility has recently
been defined in \cite{gyenis2018categorial} and studied in the context of algebraic
quantum field theory. This notion of independence is handy when it comes to systems
coming from physics, but when directly applied to classical algebras, subobject 
independence is not entirely satisfactory. The sole purpose of this note is to
introduce the notion of subalgebra independence, which is a slight variation of
subobject independence, yet this modification enables us to connect subalgebra
independence to more traditional notions of independence. Apart from drawing
connections between subalgebra independence and coproducts and congruences, we 
mainly illustrate the notion by discussing examples.}
\def\KEYWORDS{Independence, Subobject independence, Subalgebra independence}

\DeclareMathOperator{\id}{id}
\DeclareMathOperator{\triv}{triv}

\DeclareMathOperator{\Con}{Con}

\def\A{\mathfrak{A}}
\def\B{\mathfrak{B}}
\def\C{\mathfrak{C}}
\def\X{\mathfrak{X}}
\def\D{\mathfrak{D}}
\def\N{\mathfrak{N}}

\def\indep{\mathrel{\raise0.2ex\hbox{\ooalign{\hidewidth$\vert$\hidewidth\cr\raise-0.9ex\hbox{$\smile$}}}}}
\def\congindep{\indep^{\!\!c}}

\def\amalg{\oplus}

\begin{document}
	
    \title{\TITLE}
    \author{ \AUTHOR } \date{\DATE}
    \maketitle
    \thispagestyle{empty}
	
    \begin{abstract}
        \ABSTRACT
        \vspace{5mm}
    \end{abstract}


    \vspace{5mm}  \normalsize	
	


\section{Introduction}

	Specifying notions of independence of subsystems of a larger system is
	crucial in the axiomatic approach to algebraic quantum field theory. 
	It turns out that such notions of independence can be specified in a number 
	of nonequivalent ways, Summers \cite{summers1990} gives a review of the 
	rich hierarchy of independence notions; for a non-technical review of subsystem 
	independence concepts that include more recent developments as well, see 
	\cite{summers2009}.	Generalizing earlier attempts, a 
	purely categorial formulation of independence of subobjects as morphism co-possibility
	has been introduced and	studied in the recent papers
	\cite{redei2014categorial,redei2015categorial} and 
	\cite{gyenis2018categorial}. Two subobjects of
	an object are defined to be independent if any two morphisms on the
	two subobjects are jointly implementable by a single morphism on the larger object.
	More precisely, let us recall the definition from \cite{gyenis2018categorial}.
	Suppose $M$ is a class of monomorphisms and 
	$H$ is another class of morphisms of a category.
	\begin{definition}\label{def:independence-morphisms}
		$M$-morphisms $f_A:A\to X$ and $f_B:B\to X$ are	called $H$-independent 
		if for any two $H$-morphisms $\alpha:A\to A$ and $\beta:B\to B$
		there is an $H$-morphism $\gamma:X\to X$ such that the diagram below commutes.
		\begin{center}
				\begin{tikzcd}[column sep=large,row sep=large]
		A \arrow[->]{r}{f_A} \arrow{d}{\alpha} & X\arrow[dotted]{d}{\gamma}\arrow[<-]{r}{f_B} & B\arrow{d}{\beta} \\
		A \arrow[->]{r}{f_A} & X\arrow[<-]{r}{f_B} & B
				\end{tikzcd}
		\end{center}			
	\end{definition}
	The objects $A$ and $B$ can be regarded as $M$-subobjects of $X$, 
	and it is intuitively clear why $H$-independence of $M$-subobjects $A$ and $B$ is an 
	independence condition: fixing the morphism $\alpha_A$ on object $A$ does not interfere 
	with fixing any morphism $\alpha_B$ on object $B$, and vice versa. That is to say, 
	morphisms can be independently chosen on these objects seen as subobjects of object 
	$X$. 
	
	In algebraic quantum field theory, independence given by the definition above is
	specified in the context of the category of special $C^*$-algebras taken with the
	class of operations (completely positive, unit preserving, linear maps)
	between $C^*$-algebras. Considerations from physics ensure injectivity of the
	``large system'' $X$ and therefore extending morphisms from the subobjects to the 
	larger object in which independence is defined is always possible. 
	
	Although the definitions employed in \cite{gyenis2018categorial} are rather general, 
	they become too restrictive when injectivity is not guaranteed. To
	reiterate: the main concern is that independence of $A$ and $B$ should not depend 
	on whether morphisms can be extended to the entire $X$,
	but rather one should care for extensions to the subobject ``generated by'' $A$ and $B$ 
	only. In
	other words, in a concrete category of structures, independence of $A$ and $B$ should depend 
	only on how elements that can be term-defined from $A$ and $B$ relate to each other and not on 
	elements that have ``nothing to do'' with $A$ and $B$. 
	Algebraically, term-definable elements are exactly the elements of the substructure generated 
	by $A$ and $B$. Defining the notion of a generated subobject in category theoretic terms 
	in not unproblematic and we do not take the trouble here to deal with such issues. 
	Instead, we focus almost exclusively on concrete algebras or categories of algebras.
	We introduce a slight modification to Definition \ref{def:independence-morphisms}
	which makes it more useful among algebras. We illustrate this `usefulness' by examples
	where subalgebra independence coincide with well-known
	traditional notions of independence:
	\begin{itemize}\itemsep-1pt
		\item Subset independence is disjointness.
		\item Subspace independence is linear independence.
		\item Boolean subalgebra independence is logical independence.
		\item Abelian subgroup independence is the traditional notion of group independence.\footnote{However, the case of non-Abelian groups is very different.}
	\end{itemize}
	Finally, we mention a related concept that we call congruence independence.


\section{Subalgebra independence}

Let us fix an algebraic (or more generally a first order) similarity type. 
When we speak about algebras or structures, then we understand these algebras 
(structures) to have the same similarity type. We use the convention that 
algebras are denoted by Fraktur letters $\A$ and the universe of the 
algebra $\A$ is denoted by the same but capital letter $A$.
For subalgebras $\A, \B$ of $\X$ we write $\A\lor \B$ for the subalgebra of 
$\X$ generated by $A\cup B$.
\medskip

\begin{definition}[Subalgebra-independence]\label{def:subalgindep}
	Let $\X$ be an algebra and $\A, \B$ be subalgebras of $\X$. 
	We say that $\A$ and $\B$ are \emph{subalgebra-independent in} $\X$ if for any homomorphisms 
	$\alpha:\A\to\A$ and $\beta:\B\to\B$ there is a homomorphism $\gamma:\A\lor\B\to\A\lor\B$
	such that the diagram below commutes.
	\begin{center}
		\begin{tikzcd}[column sep=large,row sep=large]
			\A \arrow[->]{r}{\subseteq} \arrow{d}{\alpha} & \A \lor \B\arrow[dotted]{d}{\gamma}\arrow[<-]{r}{\supseteq} & \B\arrow{d}{\beta} \\
			\A \arrow[->]{r}{\subseteq} & \A \lor \B\arrow[<-]{r}{\supseteq} & \B
		\end{tikzcd}
	\end{center}
	The homomorphism $\gamma$ is called the joint extension of $\alpha$ and $\beta$ (to $\A\lor\B$).
	We write $\A\indep_{\X}\B$ when 
	$\A$ and $\B$ are subalgebra-independent in $\X$, 
	and we might omit the subscript $\X$ when it is clear from the context. 
		
\end{definition}
\medskip

\noindent When the algebras in question have particular names, e.g. groups, fields, etc., then we specify the independence as ``subgroup-independence", ``subfield-independence" etc. \\

Comparing subalgebra independence with Definition \ref{def:independence-morphisms} it is clear
that the inclusion mappings take the role of $M$-morphisms and $H$ is the class of all
homomorphisms between algebras. The main difference, however, is that in subalgebra
independence we extend the mappings $\alpha$ and $\beta$ to the substructure generated by
$A\cup B$ only. We also note that $H$ could be chosen differently, e.g. it could be the 
class of automorphisms, leading to variations of the notion of independence. We do not 
discuss such variations in this paper. \\

Before discussing the examples, let us state some useful propositions.
First, it is an immediate consequence of the definition of subalgebra independence that
the joint extension of $\alpha$ and $\beta$ is always unique (if exists):

\begin{proposition}\label{claim:unique}
	If the joint extension $\gamma:\A\lor\B\to\A\lor\B$ of $\alpha:\A\to\A$ and $\beta:\B\to\B$ 
	exists, then it is unique and is given by
	\[
\gamma\big( t^{\A\lor\B}(\vec{a}, \vec{b}) \big) \; = \; t^{\A\lor\B}\big( \alpha(\vec{a}), \beta(\vec{b}) \big)
	\]	
	for each term $t(\vec x,\vec y)$ and elements $\vec a\in A$, $\vec b\in B$.
\end{proposition}
\begin{proof}
	Elements of $\A\lor\B$ are of the form $t^{\A\lor\B}(\vec{a},\vec{b})$ for $\vec{a}\in A$
	and $\vec{b}\in B$. As $\gamma$ is a homomorphism that extends both $\alpha$ and $\beta$, 
	we must have
	\[
	\gamma\big( t^{\A\lor\B}(\vec{a}, \vec{b}) \big) \;=\;
	t^{\A\lor\B}\big( \gamma(\vec{a}), \gamma(\vec{b}) \big) \;=\;
	t^{\A\lor\B}\big( \alpha(\vec{a}), \beta(\vec{b}) \big).
	\]
\end{proof}

\noindent 
Let $\mathbf{K}$ be a class of similar algebras regarded as a category
with homomorphisms as morphisms. Let $\A_1,\A_2\in\mathbf{K}$ and
consider embeddings $e_i:\A_i\to\C$. Then $\C$
is a coproduct of $\A_1$ and $\A_2$ in $\mathbf{K}$ iff $\C$ has the following universal 
property with respect to $\mathbf{K}$: for any $\D\in\mathbf{K}$ and 
homomorphisms $f_i:\A_i\to\D$ there is a homomorphism $g:\C\to\D$ such that 
$f_i = g\circ e_i$ ($i=1,2$). The coproduct, if exists, is unique up to isomorphism.
If $\mathbf{K}$ is clear from the context we denote a coproduct of 
$\A_1$ and $\A_2$ by $\A_1\amalg\A_2$. In what follows, we assume that $\A$ and $\B$ are
(identified with) subalgebras of the coproduct $\A\amalg\B$.

\begin{proposition}\label{claim:coprod}
	Consider $\A$ and $\B$ as subalgebras of the coproduct $\A\amalg\B$. Then any
	pair of homomorphisms $\alpha:\A\to\A$ and $\beta:\B\to\B$ has a joint 
	extension to a homomorphism $\alpha\amalg\beta:\A\amalg\B\to\A\amalg\B$.
\end{proposition}
\begin{proof}
	From the diagram below on the left-hand side, by composing arrows, 
	one gets the diagram on the right-hand side which is a coproduct diagram. 
	Therefore a suitable $\gamma$ with the dotted arrow 
	exists and completes the proof.
	\begin{center}
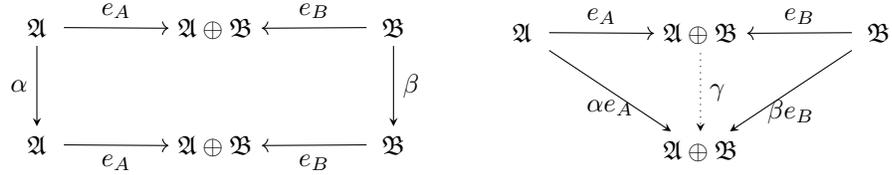
		\begin{tabular}{cc}
		\begin{tikzpicture}
		  \matrix (m) [matrix of math nodes,row sep=3em,column sep=4em,minimum width=2em]
		  {
			 \A  & \A \amalg \B  & \B\\
			 \A  & \A \amalg \B  & \B\\};
		  \path[-stealth]
			(m-1-1) edge[->] node [above] {$e_{A}$} (m-1-2)
					edge node [left]  {$\alpha$} (m-2-1)
			(m-2-1) edge[->] node [below] {$e_{A}$} (m-2-2)
			(m-1-3) edge[->] node [above] {$e_{B}$} (m-1-2)
					edge node [right] {$\beta$} (m-2-3)
			(m-2-3) edge[->] node [below] {$e_{B}$} (m-2-2);
		\end{tikzpicture}	
		& \quad
		\begin{tikzpicture}
		  \matrix (m) [matrix of math nodes,row sep=3em,column sep=4em,minimum width=2em]
		  {
			 \A  & \A \amalg \B  & \B\\
			     & \A \amalg \B  &    \\};
		  \path[-stealth]
			(m-1-1) edge[->] node [above] {$e_{A}$} (m-1-2)
					edge node [below] {$\alpha e_{A}$} (m-2-2)
			(m-1-2) edge [dotted] node [right] {$\gamma$} (m-2-2)
			(m-1-3) edge[->] node [above] {$e_{B}$} (m-1-2)
					edge node [below] {$\beta e_{B}$} (m-2-2);
		\end{tikzpicture}				
	\end{tabular}
	\end{center}
\end{proof}

\begin{proposition}\label{prop:alg-coprodindep}
	Subalgebras $\A$ and $\B$ of the coproduct $\A\amalg\B$ are subalgebra-independent
	provided $\A\lor\B = \A\amalg\B$.
\end{proposition}
\begin{proof}
	Immediate from Proposition \ref{claim:coprod}.
\end{proof}

It is clear that there is a canonical surjective homomorphism
$q:\A\amalg\B\to\A\lor\B$. Take homomorphisms 
$\alpha:\A\to\A$ and $\beta:\B\to\B$ and consider the diagram below.
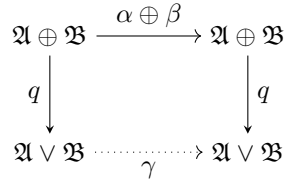
\begin{center}
	\begin{tikzpicture}
		  \matrix (m) [matrix of math nodes,row sep=3em,column sep=4em,minimum width=2em]
		  {
			 \A\amalg\B  & \A \amalg \B \\
			 \A\lor\B  & \A\lor\B\\};
		  \path[-stealth]
			(m-1-1) edge[->] node [above] {$\alpha\amalg\beta$} (m-1-2)
					edge node [left]  {$q$} (m-2-1)
			(m-1-2) edge node [right] {$q$} (m-2-2)
			(m-2-1) edge[dotted,->] node [below] {$\gamma$} (m-2-2);
	\end{tikzpicture}	
\end{center}	
Then the joint extension $\gamma:\A\lor\B\to\A\lor\B$ of $\alpha$ and $\beta$ exists
if and only if the mapping
\[  \gamma(q(x)) = q((\alpha\amalg\beta)(x)) \]
is well-defined, that is, $\alpha\amalg\beta$ is ``compatible'' with the kernel $\ker(q)$. 
We make use of this observation later on when we discuss the case of groups.

\bigskip

\noindent Let us see the examples without further ado. 


\subsection{Sets}

Sets can be regarded as structures having the empty set as similarity type. 
If $A$ and $B$ are subsets of $C$, then the subset of $C$ 
generated by $A$ and $B$ is simply their union $A\cup B$. 
It is straightforward to check that subset independence coincides with disjointness. 

\begin{proposition}
	For $A,B \subseteq C$ we have $A \indep B$ if and only if $A\cap B = \emptyset$.
\end{proposition}
\begin{proof} 
	It is straightforward to check that $A$ and $B$ are independent 
	if and only if they are disjoint as 
	otherwise one could take permutations 
	of $A$ and $B$ that act differently on the intersection 
	disallowing a joint extension of 
	these permutations to $A\cup B$. 
\end{proof}

Let $\mathbf{Set}$ be the category of sets as objects and functions as morphisms. 
The coproduct $A\amalg B$ of two sets $A$ and $B$ exists and is equal 
(isomorphic) to the disjoint union of $A$ and $B$. Hence we get the following 
corollary.

\begin{corollary}
	For subsets $A,B \subseteq C$ we have $A \indep B$ if and only if $A\cup B \cong A\amalg B$. \qed
\end{corollary}


\subsection{Vector spaces}

Let $\mathbf{Vect}_{\mathbb{F}}$ be the class (category) of vector spaces 
over the field $\mathbb{F}$. Homomorphisms between vector spaces are 
precisely the linear mappings. Recall that two subspaces $\A$ and $\B$ of a 
vector space $\C$ are linearly independent if and only if $A\cap B = \{0\}$.

We claim that subspace independence coincides with linear independence of subspaces.

\begin{proposition}
	For subspaces $\A, \B$ of a vector space $\C$ we have $\A\indep \B$ 
	if and only if $A\cap B = \{0\}$.  
\end{proposition}
\begin{proof}
	Take homomorphisms $\alpha:\A\to \A$ and $\beta:\B\to \B$. 
	Then $\alpha$ and $\beta$ act on the bases $\<a_i:i\in I\>=\A$ 
	and $\<b_j:j\in J\>=\B$. Any function defined on bases can be extended 
	to a linear mapping, therefore $\alpha$ and $\beta$ 
	have a common extension 
	\[
	\gamma: \< a_i, b_j: i\in I, j\in J \> \to \< a_i, b_j: i\in I, j\in J \>
	\]
	if and only if 
	they act on $A\cap B$ the same way. As $\alpha$, $\beta$ 
	were arbitrary, the latter condition is equivalent to $A\cap B=\{0\}$.
\end{proof}

Coproduct in the category $\mathbf{Vect}_{\mathbb{F}}$ of vector 
spaces over the fixed field $\mathbb{F}$ coincides with the direct sum construction.
Let us denote the direct sum (coproduct) of two subspaces $\A$, $\B$ by $\A\amalg \B$.

\begin{corollary}
	For subspaces $\A, \B$ of a vector space $\C$ we have $\A\indep \B$ 
	if and only if $\A\lor \B \cong \A\amalg \B$. \qed
\end{corollary}


\subsection{Boolean algebras}\label{subsec:ba}
\def\indepBA{\parallel}

Let $\mathbf{Bool}$ be the category of Boolean algebras as 
objects with Boolean homomorphisms as morphisms.
The Boolean algebra $\C$ is the internal sum of the subalgebras 
$\A,\B \leq \C$ just in case the union $A\cup B$ generates $\C$ and 
whenever $a\in A$, $b\in B$ are non-zero elements, then $a\land b$ is non-zero 
(cf. Lemma 1 on p. 428 in \cite{givant2008introduction}).
This latter condition is called \emph{Boole-independence}: two subalgebras 
$\A, \B\leq \C$ are Boole-independent ($\A\indepBA \B$ in symbols) 
if for all $a\in A$, $b\in B$ we have $a\cap b\neq 0$ provided $a\neq 0\neq b$. 

The internal sum construction coincides with the coproduct $\A\amalg \B$ 
in the category $\mathbf{Bool}$. As before $\A\lor \B$ is the subalgebra 
(of $\C$) generated by $A\cup B$. Then we have 
$\A\lor \B \cong \A\amalg \B$ precisely when $\A\indepBA \B$.\\

We claim that Boolean subalgebra independence coincides with 
Boole-independence of subalgebras.

\begin{proposition} \label{boolprop1}
	For Boolean subalgebras $\A, \B$ of a Boolean algebra $\C$ we have
	\[ \A \indep \B \quad\Longleftrightarrow\quad \A\indepBA \B 
	\quad\Longleftrightarrow\quad \A\lor \B \cong \A\amalg \B. \]
\end{proposition}
\begin{proof}
	The second equivalence 
	$\A\indepBA \B \quad\Longleftrightarrow\quad \A\lor \B = \A\amalg \B$ is clear. 
	By Proposition \ref{prop:alg-coprodindep} coproduct injections are 
	always independent, therefore we have
	\[ \A\indepBA \B \quad\Rightarrow\quad \A \indep \B.  \]
	
	As for the converse implication 
	assume $\A\indep \B$. By way of contradiction
	suppose there are non-zero elements $a\in A$, $b\in B$ 
	so that $a\cap b = 0$. For an element $x$ let $x'$ stand for the Boolean negation 
	(complement) of $x$.	
	Take a homomorphism $\alpha:\A\to \A$ 
	such that $\alpha(a) = 1\in A$ and $\alpha(a')=0\in A$ 
	(e.g. take an ultrafilter in $\A$ that contains $a$, and send 
	elements belonging to the ultrafilter to $1\in A$).
	Take $\beta = \id_B$. This two homomorphisms cannot 
	be jointly extended to a homomorphism $\gamma:\A\lor \B\to \A\lor \B$
	because such a joint extension $\gamma$ would 
	satisfy $\gamma(a') = \alpha(a') = 0$ and 
	$\gamma(b) = \beta(b) = b \neq 0$.
	As $b\subseteq a'$ it must follow that 
	$\gamma(b) \subseteq \gamma(a') = 0$; contradiction.
\end{proof}

We remark that $\A\indepBA \B$ implies $A\cap B = \{0,1\}$ (for if $0\neq a \neq 1$ was 
an element of $A\cap B$, then taking $a\in A$ and $a'\in B$ would witness non-Boole-independence).
Thus, similarly to the previous cases, subalgebra-independence requires that the two subalgebras
in question intersect in the minimal subalgebra.\\

Notice that Boolean independence coincides with logical independence
if the Boolean algebras are viewed as the Lindenbaum--Tarski algebras of a 
classical propositional logic: $a\land b\neq 0$ entails that there is an 
interpretation on $C$ that makes $a\land b$ hence both $a$ and $b$ true; 
i.e. any two propositions that are not contradictions can be jointly true 
in some interpretation. Therefore, Boolean-subalgebra independence
captures logical independence in the category $\mathbf{Bool}$.


\subsection{Abelian groups}

The category $\mathbf{AbGrp}$ contains commutative groups as objects and group homomorphisms as arrows. 
The commutative group $\mathfrak{G}$ is the internal direct sum of its two subgroups $\mathfrak{H}$ and $\mathfrak{F}$ 
if and only if $\mathfrak{G}$ is generated by $H\cup F$ and $H\cap F = \{e\}$ (here and later on, $e$ is the unit element of the group).
(Internal) direct sums are precisely the coproducts, denoted by $\A\amalg\B$, in the category $\mathbf{AbGrp}$. 

We claim that abelian-subgroup independence coincides with having the trivial group as the intersection.

\begin{proposition}\label{prop:groupintersection}
	For subgroups $\A, \B$ of the commutative group $\C$ we have
	\[ \A \indep \B \quad\Longleftrightarrow\quad A\cap B = \{e\} \quad\Longleftrightarrow\quad \A\lor \B \cong \A\amalg \B. \]
\end{proposition}
\begin{proof}
	As $\A\lor \B$ is the subgroup of $\C$ generated by $A\cup B$, the equivalence 
	\[A\cap B = \{e\} \quad\Longleftrightarrow\quad \A\lor \B \cong \A\amalg \B\] is clear. 
	Since summands of a coproducts are always independent (Proposition \ref{prop:alg-coprodindep}) 
	we also have 
	\[A\cap B = \{e\} \quad\Longrightarrow\quad \A\indep \B.\]
	As for the other direction suppose, by way of contradiction, that there is $e\neq g\in A\cap B$. 
	Take $\alpha:\A\to \A$, $\alpha(x)=e$ and $\beta = \id_B$. 
	These two homomorphisms cannot have a joint extension 
	to $\A\lor\B$ as $\alpha(g)\neq \beta(g)$; contradicting $\A\indep \B$.
\end{proof}

Independence of subgroups $\A,\B$ of $\C$ was defined in \cite{SethWarnerModernAlgebra} by the
condition $A\cap B = \{e\}$. In the case of Abelian groups, subgroup independence gives back this
exact notion, however, the case of general groups is much more complicated.


\subsection{Groups}

Consider the category $\mathbf{Grp}$ of groups with homomorphisms. Coproducts in this category exist and are isomorphic to free products. Recall that the free product of two groups is infinite and non-commutative even if both groups are finite or commutative (unless one of them is trivial as in this case the free product is isomorphic to one of the two groups).
Suppose $\A, \B\leq \C$. The proof of Proposition \ref{prop:groupintersection} shows that $\A\indep \B$ implies $A\cap B = \{e\}$.

\begin{proposition}
	If $\A\indep\B$, then $A\cap B=\{e\}$. \qed
\end{proposition}

On the other hand, consider the subgroups $\mathbb{Z}_2, \mathbb{Z}_3$ of $\mathbb{Z}_6$ (here $\mathbb{Z}_n$ is the modulo $n$ group with addition). These subgroups are independent as Abelian 
groups, and since any homomorphic image of a commutative group is commutative, they are independent as groups, too. 
But the free product (coproduct) $\mathbb{Z}_2\amalg \mathbb{Z}_3$ is infinite, thus it is not isomorphic to 
$\mathbb{Z}_2 \lor \mathbb{Z}_3 = \mathbb{Z}_6$. This is an example for an algebraic category where subalgebra independence
and being an internal coproduct are not equivalent. \\


Using the next proposition we can draw some useful sufficient conditions 
for subgroup independence.

\begin{proposition}\label{prop:mikorftln}
	$\A\indep\B$ if and only if for all homomorphisms 
	$\alpha:\A\to\A$ and $\beta:\B\to\B$ and elements $a_i\in A$, $b_i\in B$ 
	we have
	\[ \prod a_ib_i = e \;\;\text{ implies }\;\; \prod\alpha(a_i)\beta(b_i) = e 	\]
\end{proposition}
\begin{proof}
	Consider the diagram below and let $\N$ be the normal subgroup of $\A\amalg\B$
	corresponding to the kernel $\ker(q)$.
	\begin{center}
		\begin{tikzpicture}
			  \matrix (m) [matrix of math nodes,row sep=3em,column sep=4em,minimum width=2em]
			  {
				 \A\amalg\B  & \A \amalg \B \\
				 \A\lor\B  & \A\lor\B\\};
			  \path[-stealth]
				(m-1-1) edge[->] node [above] {$\alpha\amalg\beta$} (m-1-2)
						edge node [left]  {$q$} (m-2-1)
				(m-1-2) edge node [right] {$q$} (m-2-2)
				(m-2-1) edge[dotted,->] node [below] {$\gamma$} (m-2-2);
		\end{tikzpicture}	
	\end{center}	
	The joint extension $\gamma:\A\lor\B\to\A\lor\B$ of $\alpha$ and $\beta$ exists
	if and only if $(\alpha\amalg\beta)(\N)\subseteq \N$ as this is equivalent to
	that the mapping
	\[  \gamma(q(x)) = q((\alpha\amalg\beta)(x)) \]
	is well-defined. 
\end{proof}

Observe that Proposition \ref{claim:unique} implies that whenever $\alpha$ and $\beta$ has 
a joint extension $\gamma$, then $\gamma$ is given by the equation
\[ \gamma\big( \prod a_ib_i \big) \;=\; \prod \alpha(a_i)\beta(b_i) \]
for every element $\prod a_ib_i$ of $\A\lor\B$ (where $a_i\in A$, $b_i\in B$).

\begin{proposition}
	If $\A$ and $\B$ are normal subgroups, such that $A\cap B = \{e\}$, then
	$\A\indep\B$.
\end{proposition}
\begin{proof}
	If $\A$ and $\B$ are normal subgroups with $A\cap B=\{e\}$, then $ab=ba$ holds
	for all $a\in A$ and $b\in B$. For, $a(ba^{-1}b^{-1})\in A$ and 
	$(aba^{-1})b^{-1}\in B$, and thus $aba^{-1}b^{-1} \in A\cap B = \{e\}$. 
	Let us apply Proposition \ref{prop:mikorftln}. Take homomorphisms $\alpha$ and 
	$\beta$ and elements $a_i\in A$ and $b_i\in B$. Write $a=\prod a_i$ and $b=\prod b_i$.
	By the first observation $\prod a_ib_i = ab$ follows. 
	Thus, if $\prod a_ib_i =e$, then $ab=e$. 
	As $a\in A$, $b\in B$ and $A\cap B=\{e\}$, we have $a=b=e$. Therefore
	$\alpha(a)\beta(b)=e$. Using the homomorphism property and reordering the product
	we get $\prod \alpha(a_i)\beta(b_i)=e$ as desired.
\end{proof}

However, if one of the subgroups is normal but the other is not, then they cannot be
subgroup independent.

\begin{proposition} \label{ppp}
	If $\A$ and $\B$ are subgroups such that $\A$ is normal but $\B$ is not normal in their
	join, then $\A\not\indep\B$.
\end{proposition}	
\begin{proof}
	We can assume $A\cap B=\{e\}$ as this condition is necessary for subgroup independence.
	
	Note first that given the assumptions there must exist $a\in A$ and $b\in B$ such that
	$ab\neq ba$. Otherwise, we would have $aBa^{-1} = B$ for all $a\in A$, and thus
	\[ gBg^{-1} = a_1b_1...a_nb_nBb^{-1}_na^{-1}_n...b^{-1}_1a^{-1}_1 \]
	would yield $B$, contradicting $B$ being not normal. 
	
	Pick $a \in A$ and $b \in B$ with $ab \neq ba$.	Then $bab^{-1} \neq a$, 
	but $bab^{-1} \in A$ since $\A$ is a normal subgroup.
	Therefore $bab^{-1} = a' \neq a$ and $a' \in A$. 
	Let $\alpha: \A \rightarrow \A$ be the identity function and 
	$\beta: \B \rightarrow \B$ be such that $\beta(x) = e$. 
	If $\sigma$ was a joint extension of $\alpha$ and $\beta$ then we 
	would get
	\begin{eqnarray}
		\sigma(bab^{-1}) &=& \sigma(b) \sigma(a) \sigma(b^{-1}) = eae = a, \\
		\sigma(a') &=& a'.
	\end{eqnarray}
	Hence, $\sigma(bab^{-1}) \neq \sigma(a')$ which contradicts $bab^{-1} = a'$.
\end{proof}

One might be tempted to think that because normal subgroups are independent, and if exactly one of the subgroups is normal, then they are not independent, it could also be the case that two non-normal subgroups cannot be independent. Unfortunately, this is not so, as indicated by the 
example below.

\begin{example}
	Consider the group $D_{\infty}$ given by 
	the presentation $D_{\infty} = \< x, y \mid x^2=y^2=e\>$. 
	Let $A$ and $B$ be its subgroups generated respectively by $x$ and $y$. 
	Clearly $A\cong B\cong\mathbb{Z}_2$. None of $A$ and $B$ are normal 
	subgroups of $D_{\infty}$, yet $A\indep_{D_{\infty}} B$ since the only 
	homomorphisms $A\to A$ and $B\to B$ are either the 
	identical or the trivial mappings, each can be extended 
	to a joint homomorphism $D_{\infty}\to D_{\infty}$.
\end{example}

In the previous example $D_{\infty}$ is the free product of its subgroups $A$ and $B$.
The next example shows that two non-normal subgroups can be subgroup independent in
finite groups too.

\begin{example}
	Let $A = \{e, (12)\}$ and $B=\{ e, (13)(24)\}$ be subgroups of the symmetric group on four 
	elements. The subgroup generated by $A\cup B$ is isomorphic to 
	the dihedral group $D_{4}$. None of 
	$A$ or $B$ are normal subgroups, still $A\indep B$ for the same reason as in the previous
	example. 
\end{example}

We do not yet have any nice group theoretical characterization of subgroup independence
and we leave it as an open problem.


\subsection{Graphs}

Let us see a non-algebraic example.
\def\G{\mathfrak{G}}
A graph is a structure of the form $\G = (V, E)$, where $V$ is a set and $E$ is a binary relation $E\subseteq V\times V$. There are at least two different types of homomorphisms between graphs: weak and strong homomorphisms. Let us recall the definitions.

\begin{definition}
	Given two graphs $(V,E)$ and $(W,F)$ the mapping $f:V\to W$ is a (weak) homomorphism if 
	\begin{equation}
		(u,v)\in E\quad \Longrightarrow \quad (f(u), f(v))\in F,
	\end{equation}
	and a strong homomorphism, if
	\begin{equation}
		(u,v)\in E\quad \Longleftrightarrow \quad (f(u), f(v))\in F.
	\end{equation}	
\end{definition}

Subgraphs can be understood in the graph theoretic way (that is, embeddings are weak homomorphisms) or as substructures (i.e. we take inclusions as strong embeddings; this corresponds to spanned subgraphs in the graph theoretic terminology). 

Let $\mathbf{Gra}_w$ and $\mathbf{Gra}_s$ respectively be the category of graphs with weak or strong homomorphisms as arrows. In both cases the coproduct of two graphs $\G_1$ and $\G_2$ exists and is (isomorphic to) their disjoint union, denoted by $\G_1\amalg \G_2$. 
By Proposition \ref{prop:alg-coprodindep} it is clear that $\G_1\indep_{\G_1\amalg\G_2}\G_2$. But not the other way around:

\begin{example}
	Call a graph $\G$ rigid if the identity is its only (weak) homomorphism. There are arbitrarily large rigid graphs \cite{nevsetvril2012sparsity, koubek1985quotients}. Take two rigid graphs $\G_1$ and $\G_2$ such that their underlying sets are not disjoint. Then $\G_1\indep_{\G_1\cup\G_2}\G_2$ are independent, nevertheless, 
	$\G_1\cup\G_2$ is not the coproduct of $\G_1$ and $\G_2$.  
\end{example}




\section{Joint extension of congruences}

A property that is strongly related to subalgebra independence is
the joint extension property of congruences. 
Suppose $\alpha:\A\to\A$ and $\beta:\B\to\B$ are homomorphisms and
there is a joint extension $\gamma:\A\lor\B\to\A\lor\B$ such 
that the diagram in Definition \ref{def:subalgindep} commutes.
This implies a relation between the kernels of the homomorphisms:
\begin{align}
	\ker(\gamma)\cap (A\times A) = \ker(\alpha),\quad\text{ and }\quad
	\ker(\gamma)\cap (B\times B) = \ker(\beta) \label{al:elso}
\end{align}
If $\A\indep\B$, then \eqref{al:elso} is the case for all congruences that are
kernels of the appropriate endomorphisms. This motivates the following definition.

\bigskip
\begin{definition}
	Let $\X$ be an algebra and $\A, \B$ be subalgebras of $\X$. We say that
	$\A$ and $\B$ are \emph{congruence-independent in} $\X$ if for any
	congruences $\theta_A\in\Con(\A)$ and $\theta_{B}\in\Con(\B)$ there is
	a congruence $\theta\in\Con(\A\lor\B)$ such that
	\[\theta\cap (A\times A) =  \theta_A,\quad \text{ and }\quad \theta\cap (B\times B)=\theta_B\]	
	We write $\A\congindep_{\X}\B$ when $\A$ and $\B$ are congruence-independent in $\X$, 
	and we might omit the subscript $\X$ when it is clear from the context.
\end{definition}
\bigskip

Notice that $\A\congindep\B$ implies $|A\cap B|\leq 1$. For if $|A\cap B|\geq 2$,
take the two congruences $\theta_A = \id_A$ and $\theta_B = B\times B$ (or $\theta_A = A\times A$ 
and $\theta_B = \id_B$).
Then no $\theta$ can have the property
\[\theta\cap (A\times A) =  \theta_A,\quad \text{ and }\quad \theta\cap (B\times B)=\theta_B\]	
as in that case we would have
\[
\theta \cap (A\cap B)^2 = \theta_A\cap (A\cap B)^2 
= \id_{A\cap B} \neq (A\cap B)^2 = \theta_B\cap (A\cap B)^2 = \theta \cap (A\cap B)^2.
\]

\mbox{ }\bigskip

The connection between subalgebra independence and congruence independence is
subtle, and already sets show that none implies the other. Take for example
$A = \{a\}$ and $B = \{a,b\}$ as subsets of a set. Then	$A\congindep B$ but 
$A\not\indep B$ witnessed by $\alpha = \id_A$ and $\beta:B\to B$, $\beta(x) = b$.
However, a proposition similar to Proposition \ref{claim:coprod} can be formulated.

\begin{proposition}\label{claim:coprod-cong}
	Consider $\A$ and $\B$ as subalgebras of the coproduct $\A\amalg\B$. Then 
	for any	congruences $\theta_A\in\Con(\A)$ and $\theta_{B}\in\Con(\B)$ there 
	is a congruence $\theta\in\Con(\A\amalg\B)$ such that
	\[\theta\cap (A\times A) =  \theta_A,\quad \text{ and }\quad 
	\theta\cap (B\times B)=\theta_B\]
\end{proposition}
\begin{proof}
	Let $\alpha:\A\to\A/\!\theta_{A}$ and $\beta:\B\to\B/\!\theta_B$ be the
	quotient mappings. Using the universal property of the coproduct, there
	is a homomorphism $\gamma$ making the diagram below commute.
	\begin{center}
		\begin{tikzpicture}
		  \matrix (m) [matrix of math nodes,row sep=3em,column sep=4em,minimum width=2em]
		  {
			 \A  & \A \amalg \B  & \B\\
			 \A/\!\theta_A  & \A/\!\theta_A \amalg \B/\!\theta_B & \B/\!\theta_B\\};
		  \path[-stealth]
			(m-1-1) edge[->] node [above] {$e_{A}$} (m-1-2)
					edge node [left]  {$\alpha$} (m-2-1)
			(m-2-1) edge[->] node [below] {$e_{A}$} (m-2-2)
			(m-1-3) edge[->] node [above] {$e_{B}$} (m-1-2)
					edge node [right] {$\beta$} (m-2-3)
			(m-2-3) edge[->] node [below] {$e_{B}$} (m-2-2)
			(m-1-2) edge[dotted] node [right] {$\gamma$} (m-2-2);
		\end{tikzpicture}	
	\end{center}
	Then $\theta = \ker(\gamma)$ is suitable. 
\end{proof}

\begin{proposition}\label{prop:alg-coprodindep-cong}
	Subalgebras $\A$ and $\B$ of the coproduct $\A\amalg\B$ are congruence-independent
	provided $\A\lor\B = \A\amalg\B$.
\end{proposition}
\begin{proof}
	Immediate from Proposition \ref{claim:coprod-cong}.
\end{proof}

\section*{Acknowledgement}
We are grateful to the anonymous referee whose careful reading of the manuscript and helpful comments have improved the paper. Research supported in part by the Hungarian National Research, Development and Innovation Office, contract number: K-134275 and by the project no. 2019/34/E/HS1/00044 financed by the National Science Centre, Poland.

\bigskip\bigskip

\small
\begin{tabular*}{0.97\textwidth}{@{\extracolsep{\fill}} lcr}
	Alexa Gopaulsingh & Zal\'an Gyenis & \"Ovge \"Ozt\"urk \\
	Department of Logic & Department of Logic & Department of Logic \\
	E\"otv\"os Lor\'and University  & Jagiellonian University & E\"otv\"os Lor\'and University \\
\end{tabular*}
\end{document}